\documentclass{article}
\usepackage{amsmath,amsthm,amsopn,amstext,amscd,amsfonts,amssymb}
\usepackage{natbib}
\usepackage{verbatim}

\def\diag{\mathop{\rm diag}\nolimits}
\def\tr{\mathop{\rm tr}\nolimits}
\def\R{\mathop{\rm Re}\nolimits}

\def\etr{\mathop{\rm etr}\nolimits}

\newcommand{\Half}{\mbox{$\frac{1}{2}$}}

\renewenvironment{abstract}
                 {\vspace{6pt}
                  \begin{center}
                  \begin{minipage}{5in}
                  \centerline{\textbf{Abstract}}
                  \noindent\ignorespaces
                 }
                 {\end{minipage}\end{center}}
\newtheorem{thm}{\textbf{Theorem}}[section]
\newtheorem{cor}{\textbf{Corollary}}[section]
\newtheorem{lem}{\textbf{Lemma}}[section]
\theoremstyle{definition}
\newtheorem{defn}{\textbf{Definition}}[section]

\title{\huge \textbf{Doubly noncentral singular matrix variate beta distributions}}
\author{
  \textbf{Jos\'e A. D\'{\i}az-Garc\'{\i}a} \thanks{Corresponding author\newline
    }\\
  Department of Statistics and Computation \\
  25350 Buenavista, Saltillo, Coahuila, Mexico \\
  E-mail: jadiaz@uaaan.mx \\[2ex]
  \textbf{Ram\'on Guti\'errez J\'aimez} \\
  Department of Statistics and O.R. \\
  University of Granada \\
  Granada 18071, Spain \\
  E-mail: rgjaimez@ugr.es\\
}
\date{}
\begin{document}
\maketitle
\begin{abstract}
In this paper, we determine the density functions of doubly noncentral singular
matrix variate beta type I and II distributions.
\end{abstract}

\section{Introduction}

Matrix variate beta type I and II distributions have been studied by many authors
using different definitions in the nonsingular case, see \cite{k:61}, \cite{c:63},
\cite{or:64}, \cite{j:64}, \cite{s:68}, \cite{k:70}, \cite{sk:79}, \cite{c:96} and
\cite{dggj:07}, among many others. Recently, these distributions have been studied in
the singular cases by \cite{u:94}, \cite{dg:97} and \cite{dggj:08a}. Beta type I and
II distributions play a very important role in several areas of multivariate
statistics, such as canonical correlation analysis, the general linear hypothesis in
MANOVA and shape theory, see \cite{mh:82}, \cite{s:68} and \cite{gm:93}.

In the nonsingular case in particular in the nonsingular case, the doubly noncentral
density functions of matrix variate beta type I and II distributions have been
studied by diverse authors, but with special emphasis on \textbf{symmetrised} density
functions and their application to the theory of matrix variate distribution, the
multivariate Behrens-Fisher problem and the generalised regression coefficient, see
\cite{ch:80,ch:81}, \cite{da:80}, \cite{chd:86} and \cite{dggj:08a, dgjgj:06}.

Using Greenacre's definition of the symmetrised density function \citep{g:73}, in an
inverse way, we obtain the doubly noncentral \textbf{nonsymmetrised} density
functions or simply, the doubly noncentral density functions of the singular matrix
variate beta type I and II distributions, see Section \ref{sec2}. Moreover, as
particular cases we find the noncentral density function of the singular matrix
variate beta type I and II distributions, from where we resolve, indirectly, the
integral proposed by \cite{c:63}, discussed by \cite{k:70} and reconsidered in
\cite[p. 191]{f:85}, see also \cite{dggj:07}, in singular and nonsingular cases.

\section{Preliminary results}\label{sec1}

In this section we give some definitions and notations for the singular matrix
variate beta type I and II distribution. We also include two results for the
symmetrised function and invariant polynomials with matrix arguments.

\subsection{Singular beta distributions}

Consider the following definition and notation.

Let $\mathbf{C}$ be a non-negative definite $m \times m$; then
$\mathbf{C}^{1/2}(\mathbf{C}^{1/2})' = \mathbf{C}$ is a reasonable nonsingular
factorization of $\mathbf{C}$, and in particular $\mathbf{C}^{1/2}$ can be $m \times
m$ upper-triangular matrix or an $m \times m$ non-negative definite square root, see
\cite{gn:00}, \cite{sk:79} and \cite{mh:82}.

Let $\mathbf{A}$ be an $m \times m$ non-negative definite random matrix with
Pseudo-Wishart distribution with $r$ degrees of freedom and a symmetric matrix of
parameters $\mathbf{\Sigma}$. We then state that $\mathbf{A} \sim \mathcal{PW}_{m}(r,
\mathbf{\Sigma})$, $\R(r) \leq (m-1)$. If $\R(r) > (m-1)$ then $\mathbf{A}$ is said
to have a Wishart distribution, with $\mathbf{A} \sim
\mathcal{W}_{m}(s,\mathbf{\Sigma})$, see \cite[p. 82]{mh:82}, \cite{u:94} and
\cite{dg:97}.

\begin{defn}\label{defbI}
Let $\mathbf{A}$ and $\mathbf{B}$ be independent, where $\mathbf{A} \sim
\mathcal{PW}_{m}(r,\mathbf{I})$ and $\mathbf{B} \sim \mathcal{W}_{m}(s,\mathbf{I})$.
We define $\mathbf{U} = (\mathbf{A} + \mathbf{B})^{-1/2}\mathbf{A} ((\mathbf{A} +
\mathbf{B})^{-1/2})'$. Then its density function is given and denoted as (see
\cite{dg:97})
\begin{equation}\label{beta}\hspace{-.5cm}
    \mathcal{B}I_{m}(\mathbf{U};q, r/2,s/2) = c |\mathbf{L}|^{(r -m -1)/2} |\mathbf{I}_{m} -
    \mathbf{U}|^{(s -m -1)/2} (d\mathbf{U}), \ \ \mathbf{0} \leq \mathbf{U} <
    \mathbf{I}_{m}.
\end{equation}
$\mathbf{U}$ is said to have a singular matrix variate beta type I distribution, and
this is denoted as $\mathbf{U} \sim \mathcal{B}I_{m}(q,r/2,s/2)$, $\R(s)
> (m-1)$; where $\mathbf{U} = \mathbf{H}_{1}\mathbf{L} \mathbf{H}'_{1}$, with
$\mathbf{H}_{1} \in \mathcal{V}_{q,m}$; $\mathcal{V}_{q,m}= \{\mathbf{H}_{1} \in
\Re^{m \times q}| \mathbf{H}'_{1}\mathbf{H}_{1} = \mathbf{I}_{q}\}$ denotes the
Stiefel manifold; $\mathbf{L} = \diag(l_{1}, \dots,l_{q})$, $1>l_{1}> \cdots > l_{q}
> 0$; $q = m$ (nonsingular case) or $q = r < m$ (singular case);
\begin{equation}\label{cte}
    c = \frac{\pi^{(-mr + rq)/2}\Gamma_{m}[(r+s)/2]}{\Gamma_{q}[r/2]
     \Gamma_{m}[s/2]}.
\end{equation}
$(d\mathbf{U})$ denotes the Hausdorff measure on $(mq-q(q-1)/2)$-dimensional manifold
of rank-q positive semidefinite $m \times m$ matrices $\mathbf{U}$ with $q$ distinct
nonnull eigenvalues, given by (see \cite{u:94} and \cite{dg:97})
\begin{equation}\label{medida}
    (d\mathbf{U}) = 2^{-q} \prod_{i=1}^{q}l_{i}^{m-q}\prod_{i<j}(l_{i} - l_{j})
    \left(\bigwedge_{i=1}^{q}dl_{i}\right) \wedge (\mathbf{H}'_{1}d\mathbf{H}_{1}),
\end{equation}
where $(\mathbf{H}'_{1}d\mathbf{H}_{1})$ denotes the invariant measure on
$\mathcal{V}_{q,m}$ and where $\Gamma_{m}[a]$ denotes the multivariate gamma
function, this being defined as
$$
  \Gamma_{m}[a] = \int_{\mathbf{R}>0} \etr(-\mathbf{R}) |\mathbf{R}|^{a-(m+1)/2} (d\mathbf{R}),
$$
$\R(a) > (m-1)/2$ and $\etr(\cdot) \equiv \exp(\tr(\cdot))$.
\end{defn}

Similarly, we have:

\begin{defn}\label{defbII}
Let $\mathbf{A} \sim \mathcal{PW}_{m}(r,\mathbf{I})$ and $\mathbf{B} \sim
\mathcal{W}_{m}(s,\mathbf{I})$ be independent. Furthermore, let $\mathbf{F} =
\mathbf{B}^{-1/2}\mathbf{A} (\mathbf{B}^{-1/2})'$. $\mathbf{F}$ is then said to have
a singular matrix variate beta type II distribution, denoted by $\mathbf{F} \sim
\mathcal{B}II_{m}(q,r/2,s/2)$ and if $\mathbf{F} = \mathbf{H}_{1}\mathbf{G}
\mathbf{H}'_{1}$, with $\mathbf{H}_{1} \in \mathcal{V}_{q,m}$ and $\mathbf{G} =
\diag(g_{1}, \dots, g_{q})$; $g_{1} > \cdots > g_{q} >0$, its density function is
given and denoted as (see \cite{dg:97})
\begin{equation}\label{efe}
    \mathcal{B}II_{m}(\mathbf{F};q,r/2,s/2)=c   |\mathbf{G}|^{(r-m-1)/2}|\mathbf{I}_{m}+\mathbf{F}|^{-(r+s)/2}
    (d\mathbf{F}), \quad \mathbf{F} \geq \mathbf{0}.
\end{equation}
where $c$ is given by (\ref{cte}), $\R(s) > (m-1)$ and $(d\mathbf{F})$ is given in an
analogous form to (\ref{medida}).
\end{defn}

Let us now extend these ideas to the doubly noncentral case, i.e. when $\mathbf{A}
\sim \mathcal{PW}_{m}(r,\mathbf{I}, \mathbf{\Omega}_{1})$ and $\mathbf{B} \sim
\mathcal{W}_{m}(s,\mathbf{I}, \mathbf{\Omega_{2}})$. In other words, $\mathbf{A}$ has
a noncentral Pseudo-Wishart distribution with a matrix of noncentrality parameters
$\mathbf{\Omega}_{1}$ and $\mathbf{B}$ has a noncentral Wishart distribution with a
matrix of noncentrality parameters $\mathbf{\Omega}_{2}$, see \cite{dgm:97}.
Subsequently, \cite{dggj:08a} reported the following:

\begin{lem}\label{lem2}
Suppose that $\mathbf{U}$ has a doubly noncentral matrix singular variate beta type
I, which is denoted as $\mathbf{U} \sim \mathcal{B}I_{m}(q,r/2,s/2,
\mathbf{\Omega}_{1}, \mathbf{\Omega}_{2})$. Then using the notation for the operator
sum as in \cite{da:80} its \textbf{symmetrised} density function is found to be\\[2ex]
$
 dF_{s}(\mathbf{U}) =  \mathcal{B}I_{m}(\mathbf{U};q,r/2,s/2) \etr\left(-\Half(
          \mathbf{\Omega}_{1}+ \mathbf{\Omega}_{2})\right)
$
\par\indent\hfill\mbox{$\displaystyle{\times \ \
    \sum_{\kappa,\lambda; \ \phi}^{\infty}\frac{\left(\Half(r+s)\right)_{\phi}}{\left(\Half r
    \right)_{\kappa}\left(\Half s \right)_{\lambda} k! \ l!} \frac{C_{\phi}^{\kappa,\lambda}
    (\Half \mathbf{\Omega}_{1}, \Half \mathbf{\Omega}_{2}) C_{\phi}^{\kappa,\lambda} (\mathbf{U}, (\mathbf{I}-
    \mathbf{U}) )}{C_{\phi}(\mathbf{I})}(d\mathbf{U})}$},
\par\noindent
with $\mathbf{0} \leq \mathbf{U} < \mathbf{I}$, $\R(s) > (m-1)$, $(a)_{\tau}$ is the
generalised hypergeometric coefficient or product of Poch\-hammer symbols and
$C_{\phi}^{\kappa,\lambda}(\cdot,\cdot)$ denotes the invariant polynomials with
matrix arguments defined in \cite{da:80}, see also \cite{ch:80} and \cite{chd:86}.
\end{lem}

Moreover:

\begin{lem}\label{lem3}
Suppose that $\mathbf{F} \geq 0$ has a doubly noncentral singular matrix variate beta
type II, which is denoted as $\mathbf{F} \sim \mathcal{B}II_{m}(q, r/2,s/2,
\mathbf{\Omega}_{1}, \mathbf{\Omega}_{2})$. Then its \textbf{symmetrised} density function is\\[2ex]
$
  dG_{s}(\mathbf{F}) =  \mathcal{B}II_{m}(\mathbf{F};q,r/2,s/2) \etr\left(-\Half(
    \mathbf{\Omega}_{1}+ \mathbf{\Omega}_{2})\right)
$
\par\indent\hfill\mbox{$\displaystyle{\times \ \
    \sum_{\kappa,\lambda; \phi}^{\infty}\frac{\Half(r+s)_{\phi}}{\left(\Half r
    \right)_{\kappa}\left(\Half s \right)_{\lambda} k! \ l!}\frac{C_{\phi}^{\kappa,\lambda}
    (\Half \mathbf{\Omega}_{1}, \Half \mathbf{\Omega}_{2}) C_{\phi}^{\kappa,\lambda}
    ((\mathbf{I}+\mathbf{F})^{-1} \mathbf{F}, (\mathbf{I}+\mathbf{F})^{-1})}
    {C_{\phi}(\mathbf{I})}(d\mathbf{F}) },$}
\par\noindent
where $\mathbf{F} > \mathbf{0}$ and $\R(s) > (m-1)$.
\end{lem}

\subsection{Symmetrised function and invariant polynomials with matrix arguments}

Consider the follow extension of the definition given by \cite{g:73}, see also
\cite{r:75}:

\begin{defn}
The symmetrised density function of the non-negative definite matrix $\mathbf{X}:m
\times m$, which has a density function $f_{_{\mathbf{X}}}(\mathbf{X})$, is defined
as
\begin{equation}\label{fs}
    f_{s}(\mathbf{X}) = \int_{\mathcal{O}(m)}f_{_{\mathbf{X}}}(\mathbf{HXH}')(d\mathbf{H}),
    \quad \mathbf{H} \in \mathcal{O}(m)
\end{equation}
where $\mathcal{O}(m) = \{\mathbf{H} \in \Re^{m \times m}| (\mathbf{HH}') =
(\mathbf{H}'\mathbf{H}) = \mathbf{I}_{m}\}$ and $(d\mathbf{H})$ denotes the
normalised invariant measure on $\mathcal{O}(m)$, then
$$
  \int_{\mathcal{O}(m)}(d\mathbf{H}) = 1
$$
\citep[p. 72]{mh:82}.
\end{defn}

Now consider the following theorem, which generalises eq. (5.4) of \cite{da:80}, and
proof of which is given by \cite{dg:06}.

\begin{lem}\label{lem:gen5.4}Let $\mathbf{A}$, $\mathbf{B}$, $\mathbf{X}$ and $\mathbf{Y}$ be $m\times m$
symmetric matrices, then we have
\begin{eqnarray}\label{gen5.4}
    \int_{\mathcal{O}(m)}C_{\phi}^{\kappa,\lambda}(\mathbf{AH}'\mathbf{XH},\mathbf{BH}'\mathbf{YH})
    (d\mathbf{H}) = \frac{C_{\phi}^{\kappa,\lambda}(\mathbf{A},\mathbf{B})C_{\phi}^{\kappa,\lambda}
    (\mathbf{X},\mathbf{Y})}{\theta_{\phi}^{\kappa,\lambda} C_{\phi}(\mathbf{I})},
\end{eqnarray}
with
$$
  \theta_{\phi}^{\kappa,\lambda} =
  \frac{C_{\phi}^{\kappa,\lambda}(\mathbf{I},\mathbf{I})}{C_{\phi}(\mathbf{I})}.
$$
\end{lem}

\section{Doubly noncentral singular matrix variate beta distributions}\label{sec2}

Taking into account equation (\ref{gen5.4}) it is now is possible to propose an
expression for the (nonsymmetrised) density functions of doubly noncentral matrix
variate beta type I and II distributions, applying the idea of \cite{g:73} (see also
\cite{r:75}), but in an inverse way.

\begin{thm}\label{teo1}
Assume that $\mathbf{U} \sim \mathcal{B}I_{m}(q,r/2,s/2, \mathbf{\Omega}_{1},
\mathbf{\Omega}_{2})$. Then its density function is\\[2ex]
$
 dF_{_{\mathbf{U}}}(\mathbf{U}) =  \mathcal{B}I_{m}(\mathbf{U};q,r/2,s/2) \etr\left(-\Half(
          \mathbf{\Omega}_{1}+ \mathbf{\Omega}_{2})\right)
$
\par\indent\hfill\mbox{$\displaystyle{\times \ \
    \sum_{\kappa,\lambda; \ \phi}^{\infty}\frac{\left(\Half(r+s)\right)_{\phi} \
    \theta_{\phi}^{\kappa,\lambda}}{\left(\Half r \right)_{\kappa}\left(\Half s \right)_{\lambda}
    k! \ l!} C_{\phi}^{\kappa,\lambda} \left(\Half \mathbf{\Omega}_{1} \mathbf{U}, \Half
    \mathbf{\Omega}_{2}(\mathbf{I}- \mathbf{U})\right)(d\mathbf{U})
    }$}
\par\noindent
with $\mathbf{0} \leq \mathbf{U} < \mathbf{I}$, $\R(s) > (m-1)$.
\end{thm}
\begin{proof} First observe that $ \mathcal{B}I_{m}(\mathbf{U};q,r/2,s/2)$ is a
symmetric function, then $\mathcal{B}I_{m}(\mathbf{HUH'};q,r/2,s/2) =
\mathcal{B}I_{m}(\mathbf{U};q,r/2,s/2)$. Thus\\[2ex]
$
 dF_{s}(\mathbf{U}) =  \mathcal{B}I_{m}(\mathbf{U};q,r/2,s/2) \etr\left(-\Half(
          \mathbf{\Omega}_{1}+ \mathbf{\Omega}_{2})\right)\\[-2ex]
$
\begin{equation}\label{intbeta1}
    \hspace{3.5cm} \times \ \
    \sum_{\kappa,\lambda; \ \phi}^{\infty}\frac{\left(\Half(r+s)\right)_{\phi}}{\left(\Half r
    \right)_{\kappa}\left(\Half s \right)_{\lambda} k! \ l!}\int_{\mathcal{O}(m)} h(\mathbf{HUH}')
    (d\mathbf{H})(d\mathbf{U}),
\end{equation}

for a function $h$. By (\ref{gen5.4}) observe that
{\small%
$$
  \int_{\mathcal{O}(m)} C_{\phi}^{\kappa,\lambda} \left(\Half \mathbf{\Omega}_{1} \mathbf{HUH}', \Half
    \mathbf{\Omega}_{2}(\mathbf{I}- \mathbf{HUH}')\right) (d\mathbf{H}) = \frac{C_{\phi}^{\kappa,\lambda}
    (\Half \mathbf{\Omega}_{1}, \Half \mathbf{\Omega}_{2}) C_{\phi}^{\kappa,\lambda} (\mathbf{U}, (\mathbf{I}-
    \mathbf{U}) )}{\theta_{\phi}^{\kappa,\lambda}C_{\phi}(\mathbf{I})}.
$$}
Then, by applying (\ref{fs}) in an inverse way, in (\ref{intbeta1}) we have
$$
  h(\mathbf{U}) = \theta_{\phi}^{\kappa,\lambda} \ C_{\phi}^{\kappa,\lambda}
  \left(\Half \mathbf{\Omega}_{1} \mathbf{U}, \Half \mathbf{\Omega}_{2}(\mathbf{I}- \mathbf{U})\right),
$$
from where the desired result is obtained.
\end{proof}
\begin{thm}\label{teo2}
Assume that $\mathbf{F} \sim \mathcal{B}II_{m}(q, r/2,s/2,\mathbf{\Omega}_{1},
\mathbf{\Omega}_{2})$.
Then we find that its density function is\\[2ex]
$
  dG_{_{\mathbf{F}}}(\mathbf{F}) =  \mathcal{B}II_{m}(\mathbf{F};q,r/2,s/2) \etr\left(-\Half(
    \mathbf{\Omega}_{1}+ \mathbf{\Omega}_{2})\right)
$
\par\indent\hfill\mbox{$\displaystyle{\times \ \
    \sum_{\kappa,\lambda; \phi}^{\infty}\frac{\Half(r+s)_{\phi} \ \theta_{\phi}^{\kappa,\lambda}}{\left(\Half r
    \right)_{\kappa}\left(\Half s \right)_{\lambda} k! \ l!} C_{\phi}^{\kappa,\lambda}
    \left(\Half \mathbf{\Omega}_{1}(\mathbf{I}+\mathbf{F})^{-1} \mathbf{F}, \Half \mathbf{\Omega}_{2}
    (\mathbf{I}+\mathbf{F})^{-1}\right) (d\mathbf{F})
    },$}
\par\noindent
where $\mathbf{F} > \mathbf{0}$ and $\R(s) > (m-1)$.
\end{thm}
\begin{proof} The proof is parallel to that given for Theorem \ref{teo1}.
\end{proof}

In the following two corollaries we shall obtain, as particular cases, the noncentral
density functions of singular matrix variate beta type I(A),I(B), II(A) and II(B)
distributions.

\begin{cor} \label{corbetaI}
Under the hypothesis of Theorem \ref{teo1}:
\begin{description}
    \item[i)] If $\mathbf{\Omega}_{1} = \mathbf{0}$, i.e. $\mathbf{A} \sim
    \mathcal{PW}_{m}(r,\mathbf{I})$, then we obtain the noncentral singular matrix
    variate beta type I(A) distribution denoted as
    $$
       \mathbf{U} \sim \mathcal{B}I(A)_{m}(q, r/2, s/2, \mathbf{\Omega}_{2})
    $$
    Its density function is then given by\\[2ex]
    $
      dF_{\mathbf{U}}(\mathbf{U}) = \mathcal{B}I_{m}(\mathbf{U};q,r/2,s/2) \etr\left(-\Half
          \mathbf{\Omega}_{2}\right )
    $
    \par\noindent\hfill\mbox{$
        \times \ {}_{1}F_{1}\left(\Half (r+s); \Half s; \Half \mathbf{\Omega}_{2} (\mathbf{I}-
          \mathbf{U})\right)(d\mathbf{U})$}
          \par\indent
    \item[ii)] Alternatively, if $\mathbf{\Omega}_{2} = \mathbf{0}$, i.e. $\mathbf{B}
    \sim \mathcal{W}_{m}(s,\mathbf{I})$, then we obtain the noncentral singular matrix variate beta
    type I(B) distribution denoted as $\mathbf{U} \sim \mathcal{B}I(B)_{m}(q, r/2, s/2,
    \mathbf{\Omega}_{1})$, for which its density function is
    $$
      dF_{_{\mathbf{U}}}(\mathbf{U}) = \mathcal{B}I_{m}(\mathbf{U};q,r/2,s/2) \etr\left(-\Half
          \mathbf{\Omega}_{1}\right ) {}_{1}F_{1}\left(\Half (r+s); \Half s; \Half \mathbf{\Omega}_{1}
          \mathbf{U}\right)(d\mathbf{U})
    $$
\end{description}
with $\mathbf{0} \leq \mathbf{U} < \mathbf{I}$, $\R(s) > (m-1)$ and where
${}_{1}F_{1}(\cdot)$ is the hypergeometric function with matrix arguments, see
\cite[definitions 7.3.1, p. 258]{mh:82}.
\end{cor}
\begin{proof} The density functions in two items are a consequence of the basic
properties of invariant polynomials, see \cite[equations (2.1) and (2.3)]{da:79}, see
also \cite[equations (3.3) and (3.6)]{ch:80}.
\end{proof}

\begin{cor} \label{corobetaII}
Under the conditions of Theorem  \ref{teo2}:
\begin{description}
    \item[i)] if $\mathbf{\Omega}_{1} = \mathbf{0}$, i.e. $\mathbf{A} \sim
    \mathcal{PW}_{m}(r,\mathbf{I})$, then we obtain the noncentral singular matrix
    variate beta type II(A) distribution denoted as
    $$
      \mathbf{F} \sim \mathcal{B}II(A)_{m}(q, r/2, s/2, \mathbf{\Omega}_{2}),
    $$
    and its  density function is\\[2ex]
    $
      dG_{\mathbf{F}}(\mathbf{F}) = \mathcal{B}II_{m}(\mathbf{F};q,r/2,s/2) \etr\left(-\Half
          \mathbf{\Omega}_{2}\right )
    $
    \par\noindent\hfill\mbox{$ \times \
        {}_{1}F_{1}\left(\Half (r+s); \Half s; \Half \mathbf{\Omega}_{2} (\mathbf{I}
          + \mathbf{F})^{-1}\right)(d\mathbf{F})$}
    \par\indent
    \item[ii)] if $\mathbf{\Omega}_{2} = \mathbf{0}$, i.e. $\mathbf{B}
    \sim \mathcal{W}_{m}(s,\mathbf{I})$, then we obtain the noncentral singular matrix
    variate beta type II(B) distribution denoted as
    $$
      \mathbf{F} \sim \mathcal{B}II_{1}(B)_{m}(q, r/2, s/2, \mathbf{\Omega}_{1}),
    $$
    for which its  density function is\\[2ex]
        $
      dG_{_{\mathbf{F}}}(\mathbf{F}) = \mathcal{B}II_{m}(\mathbf{F};q,r/2,s/2) \etr\left(-\Half
          \mathbf{\Omega}_{1}\right )
    $
    \par\noindent\hfill\mbox{$\times \ {}_{1}F_{1}\left(\Half (r+s); \Half s;
    \Half \mathbf{\Omega}_{1}(\mathbf{I}+\mathbf{F})^{-1}\mathbf{F}\right)(d\mathbf{F})$}
\end{description}
with $\mathbf{0} \leq \mathbf{F}$ and $\R(s > (m-1))$.
\end{cor}
\begin{proof} The proof is analogous to that given for Corollary
\ref{corbetaI}.
\end{proof}

\section*{Conclusions}

\cite{ch:80}, \cite{ch:81} and \cite{da:79} have found the symmetrised doubly
noncentral density functions of the nonsingular matrix variate beta type I and II
distributions. However, the question of nonsymmetrised density functions (or simply
density functions) remained to be resolved. In this paper, by applying Greenacre's
definition of symmetrised function of \citep{g:73} in an inverse way, we respond to
these two open problems with respect to singular and nonsingular cases. Furthermore,
in another way to the method given in \cite{dggj:07} and in \cite{dggj:08a}, we
obtain the noncentral density functions of singular matrix variate beta type
I(A),I(B), II(A) and II(B) distributions, from where, implicitly, we resolve the
integral proposed by \cite{c:63}, \cite{k:70} and reconsidered in \cite[p.
191]{f:85}, see also \cite{dggj:07}, in the singular case, and in the nonsingular
one, of course, by simply taking $q = m$.

\section*{Acknowledgements} This research work was partially supported  by
CONACYT-Mexico, research grant 81512 and by IDI-Spain, grant MTM2005-09209. This
paper was written during J. A. D\'{\i}az- Garc\'{\i}a's stay as a visiting professor
at the Department of Statistics and O. R. of the University of Granada, Spain.

\end{document}